\documentclass[12pt,a4paper]{amsart}
\usepackage[usenames]{color}
\usepackage{times}
\usepackage{amsfonts}
\usepackage{amsthm}
\usepackage{amsmath}
\usepackage{psfrag}

\usepackage{times}
\usepackage{latexsym,color}
\usepackage{amssymb}
\usepackage{graphicx}

\usepackage{epsfig}

\usepackage{tikz}
\usepackage[all]{xy}
 \usepackage[hang]{caption}
 \usetikzlibrary{snakes}
\usetikzlibrary{arrows}

\advance \topmargin by -\headheight
\advance \topmargin by -\headsep
\evensidemargin \oddsidemargin
\usepackage{lineno}

\newcommand{\Z}{{\mathbb Z}}
\newcommand{\F}{{\mathbb F}}
\newcommand{\C}{{\mathbb C}}

\newtheorem{theorem}{Theorem}
\newtheorem{lemma}[theorem]{Lemma}
\newtheorem{corollary}[theorem]{Corollary}
\newtheorem{proposition}[theorem]{Proposition}

\newtheorem{example}{Example}

\newtheorem{remark}{Remark}

\begin{document}

\title[Christoffel words \& their lexicographic array ]{On Christoffel words \& their lexicographic array}

\author[Luca Q. Zamboni]{Luca Q. Zamboni}
\address{Institut Camille Jordan,
Universit\'e Claude Bernard Lyon 1,
43 boulevard du 11 novembre 1918,
69622 Villeurbanne Cedex, France}
\email{zamboni@math.univ-lyon1.fr}


\subjclass[2010]{Primary 68R15 ; Secondary 20K01, 20G40.}
\date{}

\begin{abstract} By a  Christoffel matrix we mean a $n\times n$  matrix corresponding to the lexicographic array of a Christoffel word  of length $n.$ In this note we show that if $R$ is an integral domain,  then the  product of two Christoffel matrices over $R$  is commutative and is a Christoffel matrix over $R.$  Furthermore, if a Christoffel matrix over $R$ is invertible, then its inverse is a Christoffel matrix over $R.$ Consequently,  the set $GC_n(R)$ of all $n\times n$ invertible Christoffel matrices over $R$ forms an abelian subgroup of $GL_n(R).$ The subset of 
$GC_n(R)$ consisting all invertible Christoffel matrices having some element $a$ on the diagonal and $b$ elsewhere (with $a,b \in R$ distinct) forms a subgroup $H$ of $GC_n(R).$ If $R$ is a field, then the quotient $GC_n(R)/H$ is isomorphic to $(\Z/nZ)^\times,$ the multiplicative group of integers modulo $n.$ It follows that for each finite field $F$ and each finite abelian group $G,$ there exists $n\geq 2$ and a faithful representation $G\rightarrow GL_n(F)$ consisting entirely of  $n\times n$ (invertible) Christoffel matrices over $F.$  We describe the structure of $GC_n(\Z/2\Z).$  
\end{abstract}

\maketitle

Christoffel words, first introduced by E.B. Christoffel  \cite{C} in 1875, are a special class of binary words linked to the theory of continued fractions, and  regarded as a finitary analogue of Sturmian sequences. A binary word $w$ is a {\it Christoffel word } if it is Lyndon and cyclically balanced.  By a {\it Christoffel matrix} we shall mean a square matrix representing the lexicographic array of a Christoffel word. More precisely, fix $n\geq 2$ and  let $(\Z/nZ)^\times$ denote the multiplicative group of integers modulo $n$ consisting of all integers from $\{0,1,\ldots , n-1\}$ coprime with $n.$  For each $m\in (\Z/nZ)^\times,$ let $C_m(a,b)$ denote the $n\times n$ matrix whose first row is the Christoffel word $c_m(a,b)$ of length $n,$ on distinct symbols $a$ and $b,$ beginning in $a$ (and hence ending in $b)$  and having $m$-many $a$'s and hence $(n-m)$-many $b$'s,  and whose subsequent rows consist in the other $n-1$ cyclic conjugates of $c_m(a,b)$ arranged in increasing lexicographic order (relative to the order $a<b).$ The quantity $m\in  (\Z/nZ)^\times$ is called the {\it type} of the Christoffel matrix $C_m(a,b)$ (or of the corresponding Christoffel word $c_m(a,b))$ and is equal to the number of occurrences of the first letter $a$ in $c_m(a,b).$  For example, taking $n=5$  

\begin{equation*}
C_2(a,b)=  
\begin{pmatrix}
a & b & a & b & b \\
a & b & b & a & b \\
b & a & b & a & b \\
b & a & b & b & a \\
b & b & a & b & a 
\end{pmatrix}
\end{equation*}
is the $5\times 5$ Christoffel matrix of type $2$ corresponding to the Christoffel word $ababb.$  It also follows that the reverse of $c_m(a,b)$ is the Christoffel word $c_{n-m}(b,a)$ corresponding to the Christoffel matrix $C_{n-m}(b,a).$ We remark that the notation $C_m(a,b)$ (or $c_m(a,b))$ does not make reference to the size $n$ of the matrix (or the length $n$ of the Christoffel word) in the same way that the notation $(i,j)$ for a $2$-cycle in the symmetric group $S_n$ does not make reference to the number of letters $n.$  

We begin by reviewing some basic facts concerning the lexicographic array of a Christoffel word : First, since all rows of $C_m(a,b)$ are conjugate to one another, it follows that each Christoffel matrix is, up to a permutation of the rows, a circulant matrix \cite{ks}.
We recall the following characterisation of the lexicographic array of a Christoffel word :  The lexicographic array of a binary word $w\in \{a,b\}^n$ having $m$-many $a$'s and $(n-m)$-many $b$'s with $m,n$ coprime has the {\it lexicographic constant shift property}, i.e., each row of the array is a constant shift $p$ of the preceding row if and only if it is the lexicographic array of a Christoffel word. Moreover, in this case the constant shift $p$ is equal to $m^{-1}\bmod n$ (see Theorem C in \cite{jz}).  Clearly if each row of a $n\times n$ matrix $M$ is a shift by $m^{-1}$ to the right of the preceding row, then each column of  $M$ is a shift by $m$ downward of the preceding column, and conversely. Hence each column of the matrix $C_m(a,b)$ is a shift by $m$ of the preceding column and of course the first column consists of a block of $a$'s of size $m$ followed by a block of $b$'s of size $n-m.$ In fact, this column description also characterises Christoffel matrices. To see this, suppose $M$ is a $n\times n$ matrix with the property that i) column $1$ of $M$  consists of a block of $a$'s of size $m$ followed by a block of $b$'s of size $n-m,$ and ii) each subsequent column is a shift downward by $m$ of the preceding. Then by ii), each row of $M$ is a shift to the right by $m^{-1}$ of the preceding row. In particular the rows of $M$ are all conjugate to one another. It is also easily verified that the rows of $M$ are arranged in lexicographic order relative to the order $a<b.$  In fact, any two consecutive rows differ in precisely two consecutive positions : in these positions, the upper row has $ab$ while the lower has $ba.$ Hence (by Theorem C), $M$ is the lexicographic array of a Christoffel word.   \\


Let $R$ be an integral domain. We write $0,1$ for the additive and multiplicative identity of $R.$ In the next lemma we show that Christoffel matrices over $R$ commute under multiplication and that the product of two $n\times n$ Christoffel matrices over  $R$ is a Christoffel matrix over $R$

\begin{lemma}\label{times}  Let $A=C_{m_1}(a_1,b_1)$ and $B=C_{m_2}(a_2,b_2)$ be two Christoffel matrices of size $n$ over $R$ with 
$m_1,m_2 \in  (\Z/nZ)^\times$ and $a_1,a_2,b_1,b_2 \in R.$  Put $r=\left \lceil \frac{m_1m_2}n \right \rceil$ and $s=\left \lfloor \frac{m_1m_2}n \right \rfloor.$ Then \[AB= BA  = C_{m_3}(a_3,b_3)\] where $m_3=m_1m_2\in  (\Z/nZ)^\times$ (meaning $m_1m_2=ns+m_3),$ 
\[a_3= r a_1a_2 +(m_1-r) a_1b_2 + (m_2-r) b_1a_2 + (n-(m_1+m_2)+r)b_1b_2\]
and 
\[b_3= s a_1a_2 +(m_1-s) a_1b_2 + (m_2-s) b_1a_2 + (n-(m_1+m_2)+s)b_1b_2.\]
Moreover, setting $\delta_i=a_i-b_i$ and $\sigma_i= m_ia_i + (n-m_i)b_i$ (for $i=1,2,3)$ we have  $\delta_3=\delta_1\delta_2\neq 0$ and
$\sigma_3=\sigma_1\sigma_2.$

\end{lemma}

\begin{proof}We begin by noting that as $R$ is commutative, each of $a_3$ and $b_3$ given above is invariant under exchange of indices $1$ and $2.$ 
Thus it suffices to show that $AB=C_{m_3}(a_3,b_3).$ Write $A=(a_{i,j}),$  $B=(b_{i,j})$ and  $C=(c_{i,j})= AB.$

To show that $C=C_{m_1m_2}(a_3,b_3)$, we use the column characterisation of Christoffel matrices mentioned earlier :  i) $c_{i,j}=c_{i+m_1m_2,j+1}$ for all $i,j$ ; ii) the first column of the matrix $C$ consists of $m_1m_2 \bmod n$ consecutive $a_3$ followed by $(n-m_1m_2) \bmod n$ consecutive $b_3$ ; iii) $a_3\neq b_3.$   

As for i), using the relations $a_{i,j}=a_{i+m_1,j+1}$ (and hence $a_{i,j}=a_{i+m_1m_2, j+m_2})$ and $b_{i,j}=b_{i+m_2,j+1}$ we have 
\[c_{i,j}=\sum_{k=1}^na_{i,k}b_{k,j}=\sum_{k=1}^na_{i,k}b_{k+m_2,j+1}=\sum_{k=1}^na_{i+m_1m_2,k+m_2}b_{k+m_2,j+1}=c_{i+m_1m_2,j+1}
\]
where all indices are taken modulo $n.$

We next look at column $1$ of the matrix $C.$ Fix $1\leq i \leq n$ and write row $i$ of matrix $A$ as $uv$ where $u$ is the prefix of length $m_2$ and $v$ the suffix of length $n-m_2.$ 
Then taking the dot product of row $i$ of $A$ with column $1$ of $B$ (which is made up of $m_2$ consecutive $a_2$ followed by $(n-m_2)$ consecutive $b_2)$ gives 
\[c_{i,1}=|u|_{a_1}a_1a_2 + |v|_{a_1}a_1b_2 + |u|_{b_1}b_1a_2 + |v|_{b_1}b_1b_2\]
where as usual $|x|_a$ denotes the number of occurrences of $a$ in $x.$ 
Put $r=\left \lceil \frac{m_1m_2}n \right \rceil$ and $s=\left \lfloor \frac{m_1m_2}n \right \rfloor.$
First consider the case in which $u$ is {\it rich} in $a_1,$ meaning $|u|_{a_1}=r$ ; in this case  $v$ is poor in $a_1$ and hence
\[|v|_{a_1}=\left \lfloor \frac{m_1|v|}n \right \rfloor = \left \lfloor \frac{m_1(n-m_2)}n \right \rfloor=m_1-r.\]
Similarly, as $u$ is rich in $a_1,$ it follows that $u$ is poor in $b_1$ and hence
\[|u|_{b_1}=\left \lfloor \frac{(n-m_1)|u|}n \right \rfloor =\left \lfloor \frac{(n-m_1)m_2}n \right \rfloor = m_2 -r.\]
And finally, as $u$ is poor in $b_1$ it follows that $v$ is rich in $b_1$ and hence
\[|v|_{b_1}= \left \lceil \frac{(n-m_1)|v|}n \right \rceil = \left \lceil \frac{(n-m_1)(n-m_2)}n \right \rceil  = n-(m_1+m_2) + r.\]

If on the other hand $u$ is poor in $a_1,$ then a similar calculation shows the same result except with $r$ everywhere replaced by $s.$ 
Thus if the prefix of length $m_2$ of row $i$ of $A$ is rich in $a_1$ then $c_{i,1}=a_3,$ while if it is poor in $a_1,$ then $c_{i,1}=b_3.$ Now for any $k\geq 1,$ the  prefix $a_1x_1\cdots x_{k-1}$ of length $k$  of row $1$ of $A$ is rich in $a_1$ (since the corresponding suffix of the same length is given by $x_{k-1}\cdots x_1b_1).$ Thus in particular $c_{1,1}=a_3,$ i.e., column $1$ of $C$ begins in $a_3.$  Now following \cite{jz, mrs},  any two consecutive rows of a Christoffel matrix differ precisely in two consecutive positions. When comparing lexicographically the prefixes of length $m_2$ of the rows of $A,$ the switch from rich to poor (in $a_1)$ happens when $a_{i, m_2}a_{i,m_2+1}=a_1b_1$ and $a_{i+1, m_2}a_{i+1,m_2+1}=b_1a_1,$ in other words the position $i$ in column $m_2$ of $A$ where the $a_1$-block terminates. Since each column of $A$ is a shift by $m_1$ of the previous column, in column $m_2$ of $A,$ the $a_1$-block ends in row $m_1m_2 \bmod n.$   

Finally to see that $a_3\neq b_3$ we have
\begin{align*}
\delta_3=a_3-b_3 &= (r-s)a_1a_2 + (s-r)a_1b_2 + (s-r)b_1a_2 + (r-s)b_1b_2 \\
& = a_1a_2 - a_1b_2 - b_1a_2 + b_1b_2 \\
& = (a_1-b_1)(a_2-b_2)=\delta_1\delta_2 \neq 0
\end{align*}  
since $a_1\neq b_1$ and $a_2\neq b_2.$ Finally, to see that $\sigma_3=\sigma_1\sigma_2$ it suffices to multiply out $(m_1a_1 + (n-m_1)b_1)(m_2a_2 + (n-m_2)b_2)$ and use the relations $m_1m_2=ns+m_3$ and $r-s=1.$
\end{proof}
\vspace{.2 in}
\noindent If $R$ is not an integral domain, then the product of two Christoffel matrices may not be binary - for example, taking $n=2$ over $\Z/6\Z$ we have
$C_1(2,0)C_1(3,0)$ is the zero $2\times 2$ matrix.

\vspace{.2 in}
The following proposition shows that in general Christoffel matrices are diagonalisable :

\begin{proposition}\label{diag} Let $n\geq 2$ and $m\in (\Z/nZ)^\times.$ Let $A=C_m(a,b)$ be a Christoffel matrix of size $n$ with $a,b$ belonging to an algebraically closed field $\F.$ Then the minimal polynomial $m_A(x)$ divides $(x^k-\delta^k)(x-\sigma)$ where $k$ is the order of $m,$   $\delta =a-b$ and $\sigma=ma+(n-m)b.$  Moreover, if neither $n$ nor $k$ is divisible by $\text{char}(\F)$ (so in particular in characteristic $0),$  then $A$ is diagonalisable. \end{proposition}

\begin{proof} Put $p(x)=(x^k-\delta^k)(x-\sigma).$ Since $m^k\equiv 1 \bmod n, $ by Lemma~\ref{times} we have $A^k=C_{m^k}(c,d)=C_1(c,d)$ for some $c,d\in \F$ with $c-d=\delta^k.$ It follows that $A^k-\delta^kI=(d)_{n\times n}$ the constant matrix of all $d$'s. As each column of the matrix $A-\sigma I$ sums to $0,$ we deduce that $(A^k-\delta^kI)(A-\sigma I)=(0)_{n\times n},$ i.e., $p(A)=(0)_{n\times n}.$  Thus $m_A(x)$ divides  $p(x)=(x^k-\delta ^k)(x-\sigma).$ Now assume that neither $n$ nor $k$ is divisible by $\text{char}(\F).$ If $\sigma^k\neq \delta^k,$ then  $p(x)$ has $k+1$ distinct roots, namely the $k$'th roots of $\delta^k$ (of which there are $k$) and $\sigma.$ Since $m_A(x)$ divides $p(x),$  the roots of $m_A(x)$ are each of multiplicity $1$ and hence $A$ is diagonalisable. 
If on the other hand $\sigma^k=\delta^k,$ then we claim that $A^k-\delta^kI=(0)_{n\times n}.$ In fact, setting $B=A^k-\delta^kI$ we have that $B=(d)_{n\times n}$ but also $B=A^k -\sigma^kI.$ By Lemma~\ref{times},  the row/column sum of $A^k$ is $\sigma^k,$ and therefore each row/column of $B$ sums to $0.$  
So $nd=0$ and since $\text{char}(\F)$ does not divide $n$ it follows that $d=0$ as required. Hence $m_A(x)$ divides $x^k-\delta^k$ which has $k$ distinct roots. Thus again all roots of the minimal polynomial $m_A(x)$ are simple and hence $A$ is diagonalisable.  
\end{proof}

\begin{remark}\rm
\noindent The Christoffel  matrix
\begin{equation*}  
A=\begin{pmatrix}
0 & 0 & 1  \\
0 & 1 & 0 \\
1 & 0 & 0  \\
\end{pmatrix}
\end{equation*} 
is diagonalisable in characteristic $0$ but not in characteristic $2.$  In fact, $m_A(x)=x^2-1$ which has distinct roots ($\pm 1)$ in characteristic $0.$ While  in characteristic $2,$ $m_A(x)=x^2-1=(x-1)^2$ and hence $A$ is not diagonalisable. 
\end{remark}

\begin{remark}\rm Since all $n\times n$ Christoffel matrices over the complex numbers commute under multiplication (by Lemma~\ref{times}) and are diagonalisable (by Proposition~\ref{diag}), they are  simultaneously diagonalisable, i.e., there exists a matrix $P\in GL_n(\C)$ such that $P^{-1}AP$ is diagonal for every $n\times n$ Christoffel matrix $A$ over $\C.$   The columns of $P$ constitute a basis for $\C^n$ consisting of eigenvectors of complex Christoffel matrices of size $n.$ 

\vspace{.2 in}
\end{remark}

\begin{example}\rm Fix $n\geq 2$ and let $\mathcal C_n\{0,1\}$ denote the set of all Christoffel words of length $n$ over the alphabet $\{0,1\}.$ We shall apply Lemma~\ref{times} to define an associative and commutative binary operation  on $\mathcal C_n\{0,1\},$ which as we shall see, induces an associative and commutative binary operation  on central words over the alphabet $\{0,1\}$ of length $n-2.$ By a {\it central word} we mean the central factor of a Christoffel word obtained by deleting its first and last letter. We recall that every central word is a palindrome. Let $w_1, w_2 \in C_n\{0,1\}$ be two $\{0,1\}$-Christoffel words  (say of type $m_1$ and  $m_2$ respectively) and let $A_1$ and $A_2$ be their corresponding Christoffel matrices. Regarding the alphabet $\{0,1\}$ as the field $\F_2$ and $A_1$ and $A_2$ as two Christoffel matrices over $\F_2,$ it follows from Lemma~\ref{times} that the product $A_1A_2$ is a Christoffel matrix over $\F_2$ of type $m_1m_2 \bmod n.$ The first row of  $A_1A_2$ is therefore a Christoffel word $w_3$ over $\{0,1\}$ and of type $m_1m_2.$ We define  $w_1\cdot w_2=w_3.$ It follows immediately that $\cdot$ defines an associative and commutative binary operation on $\mathcal C_n\{0,1\}.$

For example, taking $n=5$ and $w_1=00101$ and $w_2=11110$ we compute $w_1\cdot w_2$ : 

\begin{equation*}  
\begin{pmatrix}
0 & 0 & 1 & 0 & 1 \\
0 & 1 & 0 & 0 & 1 \\
0 & 1 & 0 & 1 & 0 \\
1 & 0 & 0 & 1 & 0 \\
1 & 0 & 1 & 0 & 0 
\end{pmatrix} 
\begin{pmatrix}
1 & 1 & 1 & 1 & 0 \\
1 & 1 & 1 & 0 & 1 \\
1 & 1 & 0 & 1 & 1 \\
1 & 0 & 1 & 1 & 1 \\
0 & 1 & 1 & 1 & 1 
\end{pmatrix}
= \begin{pmatrix}
1 & 0 & 1 & 0 & 0 \\
1 & 0 & 0 & 1 & 0 \\
0 & 1 & 0 & 1 & 0 \\
0 & 1 & 0 & 0 & 1 \\
0 & 0 & 1 & 0 & 1 
\end{pmatrix}
\end{equation*}
\noindent and hence  $00101\cdot 11110= 10100.$ Notice that $10^{n-1}\cdot w=w$ for all $w\in \mathcal C_n$ since the Christoffel matrix corresponding to $10^{n-1}$ is the identity matrix. 

Now let $c_1, c_2  \in \{0,1\}^{n-2}$ be two central words and put $w_1=a_1c_1b_1$ and $w_2=a_2c_2b_2$ where $\{a_i,b_i\}=\{0,1\}$ for $i=1,2.$ Then we define $c_1\cdot c_2 = c_3$ where $c_3$ is the central word corresponding $w_1\cdot w_2.$  Let us see that $c_3$ is independent of the choice of $a_i,b_i.$ Since multiplication of Christoffel words is commutative, it suffices to show that $1 c_1 0 \cdot 1 c_2 0$ and $0 c_1 1 \cdot 1 c_2 0$ share a common central factor. 
So let $w= 1 c_1 0 \cdot 1 c_2 0$ and $w'= 0 c_1 1 \cdot 1 c_2 0.$ 
To see that $w$ and $w'$ share a common central factor, it suffices to show that $w$ and $w'$ are reverses of one another. 
We begin by replacing each Christoffel word by its corresponding  Christoffel matrix as follows : $ 1 c_1 0 \leftrightarrow C_{m_1}(1,0),$ $0 c_1 1 \leftrightarrow C_{n-m_1}(0,1)$ and $1 c_2 0 \leftrightarrow C_{m_2}(1,0).$ 
Now multiplying the corresponding Christoffel matrices and applying Lemma~\ref{times} we obtain 
\[C_{m_1}(1,0)C_{m_2}(1,0) = C_{m_1m_2}(r \bmod 2, s \bmod 2)\]
and
\[C_{n-m_1}(0,1)C_{m_2}(1,0) = C_{n-m_1m_2}(s \bmod 2, r \bmod 2)\]
where $r=\left \lceil \frac{m_1m_2}n \right \rceil$ and $s=\left \lfloor \frac{m_1m_2}n \right \rfloor.$ 
So $w$ corresponds to $C_{m_1m_2}(r \bmod 2, s \bmod 2)$ while  $w'$  corresponds to $C_{n-m_1m_2}(s \bmod 2, r \bmod 2).$ It  follows that $w'$ is the reverse of $w.$
\end{example}

\vspace{.2 in}

We now consider the inverse of an invertible Christoffel matrix over an integral domain $R.$ The matrix $C_m(a,b)$ is invertible if and only if  $\det C_m(a,b)$ is a unit in $R.$ Since every Christoffel matrix $C_m(a,b)$ is a row permutation of a circulant $n\times n$ matrix,  it follows that  
\begin{equation}\label{determinant}\det C_m(a,b)=\pm (ma +(n-m)b)(a-b)^{n-1}\end{equation} (see \cite{h, ks}). In particular, this implies that the row (and column) sum $ma + (n-m)b \neq 0.$ If $R$ is a field, then this is both a necessary and sufficient condition for $C_m(a,b)$ to be invertible, i.e., over a field $F$ the Christoffel matrix $C_m(a,b)$ is invertible if and only if $ma + (n-m)b\neq 0.$

\noindent We now show that the inverse of an invertible Christoffel matrix of type $m$ is a  Christoffel matrix of type $m^{-1}$ :

\begin{lemma}\label{inverse} Let $C_m(a,b)$ be an invertible Christoffel matrix over an integral domain $R.$  Then its inverse is a Christoffel matrix over $R.$ More precisely, define $c,d,e,f \in R$ by  \[C_m(a,b)C_{m^{-1}}(1,0)=C_1(c,d)\] and 
\[C_m(a,b)C_{m^{-1}}(0,1)=C_1(e,f).\] Then $cf-ed$ is a unit in $R$ and 
\begin{equation} \label{e-inverse} C_m(a,b)^{-1}=C_{m^{-1}}(f(fc-ed)^{-1}, -d(fc-ed)^{-1}).\end{equation}
\end{lemma} 

\begin{proof}Assume $C_m(a,b)$ is invertible, i.e.,  $\det C_m(a,b)$ is a unit in $R.$ We begin by showing that $cf-ed$ is a unit in $R.$ Let $w=c_m(a,b)$ denote the Christoffel word corresponding to the first row of $C_m(a,b).$ Writing $w=uv$ with $|u|=m^{-1}$ (and hence $|v|=n-m^{-1})$, then as in the proof of Lemma~\ref{times} we have
\[c=|u|_a a +|u|_b b \,\,\,\,\,\,\,\,\,\,\,\, d=(|u|_a-1) a + (|u|_b +1) b\]
\[ e= |v|_a a +|v|_b b \,\,\,\,\,\,\,\,\,\,\,\, f= (|v|_a+1) a + (|v|_b -1) b.\]
This gives 
\begin{align*}
cf-ed & = (|u|_a + |v|_a) a^2 -(|u|_a-|u|_b + |v|_a-|v|_b) ab - (|u|_b+|v|_b) b^2 \\
& = ma^2 +(n-2m)ab -(n-m)b^2\\
& = (ma + (n-m)b)(a-b).
\end{align*}

It follows from (\ref{determinant}) that $cf-ed$ divides $\det C_m(a,b)$ which by assumption is a unit in $R,$ and hence $cf-ed$ is a unit in $R.$ 
 
 Next we claim that $f+d\neq 0.$ In fact, suppose to the contrary that $f+d=0.$ By Lemma~\ref{times} $c-d=a-b=f-e,$ and hence $c+e=0.$
 Thus
 \[C_m(a,b) \left (C_{m^{-1}}(1,0) + C_{m^{-1}}(0,1) \right ) = C_1(c,d) + C_1(e,f) = C_1(c,d) + C_1(-c,-d) = (0)_{n\times n}.\]
 But $C_{m^{-1}}(1,0) + C_{m^{-1}}(0,1)$ is the constant matrix $(1)_{n\times n}$  and hence 
 $C_m(a,b) (1)_{n\times n}=(0)_{n\times n}$ which is a contradiction since $C_m(a,b)$ is assumed to be invertible. 
 
 We have thus far that the matrix on the right hand side of (\ref{e-inverse}) is well defined (meaning that $cf-ed$ is a unit) and a Christoffel matrix (meaning that $f(fc-ed)^{-1} \neq -d(fc-ed)^{-1}).$
 We now show that this matrix is in fact the inverse of $C_m(a,b).$ First suppose that $df=0.$ Then either $d=0$ or $f=0.$
 If $d=0,$ then 
 \[C_m(a,b) C_{m^{-1}}(1,0)= C_1(c,0)= cI_{n\times n}\]
 where $I_{n\times n}$ denotes the identity matrix over $R.$ Also as $d=0$ we have that 
 $cf$ is a unit (and hence $c$ is a unit) and  $C_m(a,b)^{-1}= C_{m^{-1}}(c^{-1}, 0)$ which is a special case of equation (\ref{e-inverse}) when $d=0.$ Similarly if $f=0$ then we find that $C_m(a,b)^{-1}= C_{m^{-1}}(0, e^{-1})$  as required.
 Finally if $df\neq 0,$ then 
 \[C_m(a,b)C_{m^{-1}}(f,0)=C_1(fc,fd)\] and
 \[ C_m(a,b)C_{m^{-1}}(0,-d)=C_1(-ed,-fd)\] and therefore 
 \[C_m(a,b) \left ( C_{m^{-1}}(f,0) + C_{m^{-1}}(0,-d) \right ) = C_1(fc-ed, 0)\] or equivalently 
 \[C_m(a,b) C_{m^{-1}}(f,-d) = C_1(fc-ed, 0).\] 
 Hence \[C_m(a,b)^{-1}=C_{m^{-1}}(f(cf-ed)^{-1}, -d(fc-ed)^{-1}).\]
 \end{proof}
 
The previous lemma allows us to define an involution on (word) isomorphism classes of Christoffel words of length $n\geq 2.$ Let $\omega$ be a Christoffel class of length $n$ and of type $m\in (\Z/nZ)^\times.$ Pick a representative $w$ in $ \{0,1\}^n$ having odd many $1$'s (clearly $m$ and $n-m$ cannot both be even). Then the corresponding Christoffel matrix $C$ is invertible over the field $\F_2= \{0,1\}$ and its inverse $C^{-1}$ is a Christoffel matrix of type $m^{-1}$ corresponding to some Christoffel word $w'$ over the alphabet $\{0,1\}.$  The isomorphism class of $w'$ is then  the so-called {\it dual} of $\omega$ in the sense of \cite{bdr}.

 \begin{corollary}\label{main} Let $R$ be an integral domain and $n\geq 2. $ Then the set $GC_n(R)$ of all invertible $n\times n$ Christoffel matrices over $R$ is an abelian subgroup of $GL_n(R).$ The subset $H$ of $GC_n(R)$ consisting of all type $1$ invertible Christoffel matrices is a subgroup of $GC_n(R).$ If $R$ is a field, then the quotient $GC_n(R)/H$ is isomorphic to $ (\Z/nZ)^\times.$
 \end{corollary}
 
 \begin{proof} It follows from lemmas~\ref{times} and \ref{inverse} that the set of all $n\times n$ invertible Christoffel matrices over $R$ is closed under multiplication and inverse and hence constitutes an abelian subgroup of $GL_n(R).$ Define $f : GC_n(R) \rightarrow (\Z/nZ)^\times$ by $f(C_m(a,b))=m.$ It follows immediately from Lemma~\ref{times} that $f$ is a group homomorphism. The kernel of $f$ consists precisely of all type $1$ invertible matrices over $R$ and hence $\mbox{Ker}(f)=H$ from which it follows that $H$ is a subgroup of $GC_n(R).$

Now if $R$ is a field, then  $f$ is surjective; to see this, it suffices to show that for each $m\in (\Z/nZ)^\times,$ either $C_m(1,0)$ or $C_m(0,1)$ (or both) is invertible. In fact, suppose that $C_m(1,0)$ is not invertible. This implies that the row sum $m1=0.$ Since $m$ and $n-m$ are coprime, it follows from Bézout's identity that $(n-m)1\neq 0$ and hence $C_m(0,1)$ is invertible. It follows that if $R$ is a field, then  $GC_n(R)/H \simeq  (\Z/nZ)^\times.$
 \end{proof}
 
Now assume that $F$ is a finite field. Then as $ (\Z/nZ)^\times$ is isomorphic to a quotient of the finite abelian group $GC_n(F),$ it follows that $GC_n(F)$ contains a subgroup isomorphic to $ (\Z/nZ)^\times.$ In particular, for each $n\geq 2$ the group $ (\Z/nZ)^\times$ admits a faithful representation in terms of $\varphi(n)$-many invertible $n\times n$ Christoffel matrices over $F.$ Since every finite abelian group is a subgroup of $ (\Z/nZ)^\times$ for some $n,$ we have 
 
 \begin{corollary} Let $F$ be a finite field and $G$ a finite abelian group. Then there exists $n\geq 2$ and a faithful representation $G\rightarrow GL_n(F)$ of degree $n$ over $F$ consisting entirely of invertible Christoffel matrices over $F.$ In particular  (taking $F=\F_2=\{0,1\}), $ $G$ admits a linear representation over $\F_2$  consisting of $|G|$-many distinct $n\times n$ invertible Christoffel matrices over  $\{0,1\}.$
 \end{corollary}
 
\noindent In what follows we compute the general Christoffel group over the field $\F_2$ :  


\begin{proposition}\label{neven}Let $n\geq 2$ and write $n=2^kp_1^{\alpha_1}\cdots p_t^{\alpha_t}$   with $k, \alpha_i\geq 0$ and $p_i$ distinct odd primes. Then 
\begin{equation} \label{dproduct}GC_n(\F_2)\simeq GC_{2^k}(\F_2) \times GC_{p_1^{\alpha_1}}(\F_2) \times \cdots \times GC_{p_t^{\alpha_t}}(\F_2) . \end{equation}
Moreover, $GC_{2^k}(\F_2)\simeq \Z/2\Z \times \Z/2^{k-1}\Z$ for each $k\geq 1$ while  $GC_{p_i^{\alpha_i}}(\F_2)\simeq  (\Z/p_i^{\alpha_i}Z)^\times$ for each $1\leq i\leq t.$
\end{proposition}

\begin{proof}
First assume $n$ is odd, i.e., $ n= p_1^{\alpha_1}\cdots p_t^{\alpha_t}.$ As in Corollary~\ref{main}, let $H$ denote the subgroup of $GC_n(\F_2)$ consisting all type $1$ invertible Christoffel matrices. Then $H=\{C_1(1,0)\}$ since $C_1(0,1)$ has row sum $n-1\equiv 0 \bmod 2$ and hence is not invertible. 
It follows from Corollary~\ref{main} that $GC_n(\F_2) \simeq  (\Z/nZ)^\times.$ Thus for $n$ odd,  we have 
\[GC_n(\F_2)\simeq (\Z/nZ)^\times \simeq  (\Z/p_1^{\alpha_i}Z)^\times \times \cdots \times (\Z/p_t^{\alpha_t}Z)^\times \simeq GC_{p_1^{\alpha_1}}(\F_2) \times \cdots \times GC_{p_t^{\alpha_t}}(\F_2)\] as required.

Next assume $n$ is even and write $n=2^kn'$ with $k\geq 1$ and $n'= p_1^{\alpha_1}\cdots p_t^{\alpha_t}.$ In view of the odd case, in order to establish  (\ref{dproduct}) it suffices to show that $GC_n(\F_2) \simeq GC_{2^k}(\F_2) \times (\Z/n'Z)^\times.$  We note that for $n$  even,  the order of $GC_n(\F_2)$ is twice the order of $(\Z/n\Z)^\times.$ 
This is because when $n$ is even, the Christoffel matrices $C_m(0,1)$ and $C_m(1,0)$ are both invertible (for each $m\in (\Z/nZ)^\times).$ For  $m\in (\Z/nZ)^\times,$ put  $a_m=C_m(1,0)$ and $b_m=C_m(0,1).$ Applying Lemma~\ref{times} we find :

        \begin{equation}\label{prod1}a_{m_1}a_{m_2}=b_{m_1}b_{m_2} = \begin{cases}
  a_{m_1m_2}  &  \text{if $\left \lfloor \frac{m_1m_2}n \right \rfloor$ is even} \\
 b_{m_1m_2}  &  \text{if $\left \lfloor \frac{m_1m_2}n \right \rfloor$ is odd}
\end{cases}\end{equation}
and    
  \begin{equation}\label{prod2}a_{m_1}b_{m_2}=b_{m_1}a_{m_2} = \begin{cases}
  b_{m_1m_2}  &  \text{if $\left \lfloor \frac{m_1m_2}n \right \rfloor$ is even} \\
 a_{m_1m_2}  &  \text{if $\left \lfloor \frac{m_1m_2}n \right \rfloor$ is odd}
\end{cases}\end{equation}                
 \noindent where as always the indices are taken modulo $n$ in $(\Z/nZ)^\times.$ Thus for $n$ even,  $GC_n(\F_2)=\{a_m,b_m\}_{m\in (\Z/nZ)^\times}$ subject to the above relations. We notice that $a_1$ is the identity and $b_1, a_{n-1}$ and $b_{n-1}$ are each elements  of order $2$ from which it follows that $CG_n(\F_2)$ is never cyclic for $n>2$ even.   
 
 It is easily checked by induction that for each $j\geq 2,$ 

 \begin{equation}\label{powers}(a_m)^j=(b_m)^j = \begin{cases}
  a_{m^j}  &  \text{if $\left \lfloor \frac{m^j}n \right \rfloor$ is even} \\
 b_{m^j}  &  \text{if $\left \lfloor \frac{m^j}n \right \rfloor$ is odd}.
\end{cases}\end{equation} 

It will be convenient to have a general formula for the product $x_{m_1}y_{m_2}$ with $x,y \in \{a,b\}$ and $m_1,m_2\in (\Z/nZ)^\times.$ To this end, we consider $(\{a,b\},*)$ as the cyclic group of order two with $a$ being the identity and $*$ the group law. Thus $a*a=b*b=a$ and $a*b=b*a=b.$ Also let $i:\{a,b\}\rightarrow \{a,b\}$ denote the involution which exchanges $a$ and $b.$ Then it is easily verified that $i(x*y)=i(x)*y=x*i(y)$ for all $x,y\in \{a,b\},$ from which it follows that 
\begin{equation}\label{ieq} i^{q_1+q_2}(x*y) =i^{q_1}(x)*i^{q_2}(y) \end{equation}
for all nonnegative integers $q_1,q_2$ and all $x,y\in \{a,b\}.$ 

\begin{lemma}\label{L0} Let $x_{m_1},y_{m_2}\in GC_n(\F_2).$ Then 
\begin{equation}\label{iidentity} x_{m_1}y_{m_2}=i^q(x*y)_r \end{equation}
where $m_1m_2=qn+r$ with $r\in (\Z/nZ)^\times.$
\end{lemma}

\begin{proof} Writing $m_1m_2=qn+r$ with $r\in (\Z/nZ)^\times$ so that $q=\left \lfloor \frac{m_1m_2}n \right \rfloor,$ and using (~\ref{prod1}) and (~\ref{prod2}) we see that if $x=y,$ then
 \[x_{m_1}y_{m_2}= \begin{cases}
  a_r  &  \text{if $q$ is even} \\
 b_r  &  \text{if $q$ is odd}
\end{cases}\] 
On the other hand if $x=y$ we have
 \[i^q(x*y)_r=i^q(a)_r= \begin{cases}
  a_r  &  \text{if $q$ is even} \\
 b_r  &  \text{if $q$ is odd}
\end{cases}\] 
A similar calculation handles the case when $x\neq y.$
\end{proof}

\noindent Thus $GC_n(\F_2)=\{x_m\,|\, x\in \{a,b\},\, m\in (\Z/n'\Z)^\times\}$ and the group law is given by (\ref{iidentity}).

\begin{lemma}\label{Psi} Let \[\Psi : GC_n(\F_2) \rightarrow GC_{2^k}(\F_2) \times (\Z/n'Z)^\times\] be given by 
\[x_m \mapsto (i^q(x)_r, r')\]
where $m=q2^k+r$ and $m=q'n'+r'$ with $r\in (\Z/2^k\Z)^\times $ and $r'\in (\Z/n'\Z)^\times.$ Then $\Psi$ is a group isomorphism.   
\end{lemma} 

\begin{proof}We begin by showing that $\Psi$ is a group homomorphism. For this it suffices to show that the restriction of $\Psi$ to each of the two coordinates is a group homomorphism. By Lemma~\ref{times} the restriction of $\Psi$ to the second coordinate is a group homomorphism as the type of the product of two Christoffel matrices of size $n$ is just the product of the types taken modulo $n.$ So let \[\Psi _1 : GC_n(\F_2) \rightarrow GC_{2^k}(\F_2) \] be given by $\Psi_1(x_m)=i^q(x)_r$ where $m=q2^k+r$ with $r\in (\Z/2^k\Z)^\times .$ We show $\Psi_1$ is a group homomorphism. Let $x_{m_1},y_{m_2}\in GC_n(\F_2)$ and write $m_1m_2=qn+r,$ $m_1=q_12^k+r_1,$ $m_2=q_22^k+r_2,$ $r_1r_2=q_32^k+r_3$ and $r=q'2^k + r'$ with $r\in  (\Z/nZ)^\times$ and $r_1,r_2,r_3,r'\in (\Z/2^kZ)^\times.$
It follows immediately that $r'=r_3.$ Applying Lemma~\ref{L0} we have that $x_{m_1}y_{m_2}=i^q(x*y)_r.$
Therefore \[\Psi_1(x_{m_1}y_{m_2})=\Psi_1(i^q(x*y)_r)=i^{q'}(i^q(x*y))_{r'}=i^{q+q'}(x*y)_{r'}.\]
On the other hand $\Psi_1(x_{m_1})=i^{q_1}(x)_{r_1}$ and $\Psi_1(y_{m_2})=i^{q_2}(y)_{r_2}$ and therefore
\[\Psi_1(x_{m_1})\Psi_1(y_{m_2})=i^{q_1}(x)_{r_1}i^{q_2}(y)_{r_2}=i^{q_3}(i^{q_1}(x)*i^{q_2}(y))_{r_3}=i^{q_3}(i^{q_1+q_2}(x*y))_{r_3}=i^{q_1+q_2+q_3}(x*y)_{r_3}\]
where equation (\ref{ieq}) was used for the penultimate equality. Having already established that $r'=r_3,$ to show that $\Psi_1(x_{m_1}y_{m_2})=\Psi_1(x_{m_1})\Psi_1(y_{m_2})$ it remains to show that $q+q'\equiv q_1+q_2+q_3 \bmod 2.$ However,\[m_1m_2=qn+r=q(\frac{n}{2^k})2^k+q'2^k+r'=(q(\frac{n}{2^k}) +q')2^k+r'.\]
On the other hand,
\[m_1m_2=(q_12^k+r_1)(q_22^k+r_2)=(q_1q_22^k+q_1r_2+q_2r_1)2^k+r_1r_2=(q_1q_22^k+q_1r_2+q_2r_1 +q_3)2^k +r_3 \]
from which it follows that \[q(\frac{n}{2^k}) +q'=q_1q_22^k+q_1r_2+q_2r_1 +q_3.\]
Now since  $\frac{n}{2^k}, r_1$ and $r_2$ are each odd, reducing modulo $2$ we find that $q+q'\equiv q_1+q_2+q_3 \bmod 2$ as required.

Having shown that $\Psi$ is a group homomorphism, let us show that $\Psi$ is injective. Assume $x_m\in \mbox{Ker}\,\Psi.$ Then $\Psi(x_m)=(a_1,1)$ where $a_1$ here denotes the identity Christoffel matrix of size $2^k.$ Then writing $m=q2^k+r$ we have $a=i^q(x)$ and $r=1.$ Also $m=q'n'+1.$ Thus as $m\equiv 1 \bmod 2^k$ and $m\equiv 1 \bmod n'$ it follows that $m\equiv 1 \bmod n$ and hence $m=1$ and $q=0.$ This implies that $a=i^q(x)=i^0(x)=x$ and hence $x_m=a_1$  where this time $a_1$ denotes the identity Christoffel matrix of size $n.$ Hence $\Psi$ is injective. To see that $\Psi$ is surjective, it suffices to show that $\mbox{Card}\, GC_n(\F_2)= \mbox{Card}\, (GC_{2^k}(\F_2) \times (\Z/n'Z)^\times).$ In fact, $\mbox{Card}\, GC_n(\F_2)=2\,\mbox{Card}\,(\Z/n'\Z)^\times$ while 
\begin{align*}\mbox{Card}\, (GC_{2^k}(\F_2) \times (\Z/n'Z)^\times) &=\mbox{Card}\,\, GC_{2^k}(\F_2)\, \mbox{Card}\, (\Z/n'Z)^\times\\&=2\,\mbox{Card}\, (\Z/2^kZ)^\times \, \mbox{Card}\, (\Z/n'Z)^\times\\&=2\,\mbox{Card}\,(\Z/n\Z)^\times
\end{align*}
as required. This completes our proof of the lemma. \end{proof} 

\noindent It remains to compute $GC_{2^k}(\F_2).$ 

\begin{lemma}\label{L1}$GC_{2^k}(\F_2)\simeq \Z/2\Z \times \Z/2^{k-1}\Z\,$ for each $k\geq 1.$ \end{lemma}  

\begin{proof} Recall that for $k\geq 3,$  $(\Z/2^k\Z)^\times$ is isomorphic to $\Z/2\Z \times \Z/2^{k-2}\Z$ (a result due to Gauss) and  is generated by $2^k-1$ (of order 2) and $3$ of order $2^{k-2}$ (see for example \cite{ir}). In particular, the order of $(\Z/2^k\Z)^\times$ is $2^{k-1}$ and hence the order of $GC_{2^k}(\F_2)=2^k.$ We begin by proving the lemma for  $k=1,2.$ For $k=1$ we have $GC_2(\F_2)=\{a_1,b_1\}$ and is hence isomorphic to $\Z/2\Z$ as required. For $k=2$ we have $GC_4(\F_2)=\{a_1,b_1,a_3,b_3\}$ and since $GC_4(\F_2)$ is not cyclic, it follows that $GC_4(\F_2)\simeq \Z/2\Z \times \Z/2\Z$ as required. Assume henceforth that $k\geq 3.$ We will show that $GC_{2^k}(\F_2)$ is isomorphic to  $\Z/2\Z \times \Z/2^{k-1}\Z$ and generated by the elements $a_{2^k-1}$ of order $2$ and $a_3$ of order $2^{k-1}.$

We begin by showing (by induction on $k)$ that for each $k\geq 3$ we have  $3^{2^{k-2}}=q_k2^k+1$ with $q_k$ odd.   If $k=3,$ then $3^2=1\cdot 2^3 + 1$ so $q_3=1.$ Now assume $q_k=\frac{3^{2^{k-2}}-1}{2^k}$ is odd and let's show that 
 $q_{k+1}$ is odd. Writing
 \[q_{k+1}=\frac{3^{2^{k-1}}-1}{2^{k+1}}=\frac {(3^{2^{k-2}}-1)}{2^k} \cdot \frac {(3^{2^{k-2}}+1)}{2}.\]
 Now the first factor on the right hand side is odd by induction hypothesis. As for the second factor, we have
 \[\frac{3^{2^{k-2}}+1}{2}=\frac{q_k2^k+1+1}{2}=\frac{q_k2^k+2}{2}=q_k2^{k-1}+1 \equiv 1 \bmod 2\]
 as required. 
It follows from (\ref{powers})  that
$(a_3)^{2^{k-2}}=b_1,$ and hence $a_3$  is  of order $2^{k-1}$ in $GC_{2^k}(\F_2)$ which is half the cardinality of $GC_{2^k}(\F_2).$ Since $GC_{2^k}(\F_2)$ is not cyclic, there can be no element of higher order. Finally as noted earlier, $a_{2^k-1}\in GC_{2^k}(\F_2)$ is of order $2$ and furthermore is not contained in the subgroup $<a_3>$ of $GC_{2^k}(\F_2)$ generated by $a_3,$ for otherwise $2^k-1$ would belong to the subgroup $<3>$ of  $(\Z/2^k\Z)^\times $ generated by $3,$ a contradiction. Thus $<a_3>\,\simeq \Z/2^{k-1}\Z$ and the quotient $GC_{2^k}(\F_2)/<a_3>\, \simeq \Z/2\Z.$
Moreover, as $GC_{2^k}(\F_2)$ contains an element of order $2$ not in $<a_3>$ (namely $a_{2^k-1}),$ it follows that $ GC_{2^k}(\F_2) \simeq \Z/2\Z \,\, \times<a_3> \, \simeq \Z/2\Z \times  \Z/2^{k-1}\Z.$
\end{proof}

\noindent This completes our proof of Proposition~\ref{neven}. \end{proof}

\noindent We end with a few examples of representations of $(\Z/n\Z)^\times$ by Christoffel matrices in $GC_n(\F_2)$  : \\

 
 
\begin{example}\rm Let $n=20.$ Using the generators $3$ and $19$ of $(\Z/20\Z)^\times$ we construct different Christoffel representations of $(\Z/20\Z)^\times$ contained in $GC_{20}(\F_2).$ For example:  \[< a_3, a_{19}> =\{a_1, a_3, b_7, a_9,  a_{11}, b_{13}, a_{17}, a_{19}\}\]  

\[< a_3, b_{19}> =\{a_1, a_3, b_7, a_9, b_{11}, a_{13}, b_{17}, b_{19}\}\]

\[< b_3, a_{19}> =\{a_1, b_3, a_7, a_9, a_{11}, a_{13}, b_{17}, a_{19}\}\]

\[< b_3, b_{19}> =\{a_1, b_3, a_7, a_9, b_{11}, b_{13}, a_{17}, b_{19}\}\]
Notice that in each representation we have a Christoffel matrix of each type $m \in (\Z/20\Z)^\times.$
\end{example}

\begin{example}\rm Let $n=8.$ Then $GC_8(\F_2)=\{a_1,b_1,a_3,b_3,a_5,b_5,a_7,b_7\}$ and is generated by $a_3$ and $a_7.$  Also,  $(\Z/8Z)^\times,$ which is isomorphic to $\Z/2\Z\times \Z/2\Z,$ is isomorphic to the subgroup $\{a_1,b_1,a_7,b_7\}$  and moreover this is the unique subgroup of $GC_8(\F_2)$ isomorphic to $(\Z/8Z)^\times.$  The corresponding Christoffel words of length $8$ representing $(\Z/8Z)^\times$ are  $\{10^7, 01^7, 1^70, 0^71\}.$ Notice only types $1$ and $7$ occur. In general, for $n\equiv 0\bmod 8$ even, we obtain interesting Christoffel representations of  $(\Z/n\Z)^\times$ over $\F_2.$ By interesting we mean that the representation does not necessarily include a Christoffel matrix of each type $m\in (\Z/n\Z)^\times$ as in the case of $n=8.$   \end{example}

 {\bf Postscript :} A few days after posting the original version of this note on arXiv, we were informed of an interesting  new paper by Christophe Reutenauer and Jeffrey Shallit~\cite{rs} which contains some of the same results in addition to other related results.

\end{document}